\documentclass[article]{elsarticle}

\usepackage{lineno,hyperref}
\modulolinenumbers[5]

\usepackage[utf8]{inputenc}
\usepackage{amsmath}
\usepackage{amssymb}
\usepackage{amsthm}
\usepackage{color}
\usepackage{tikz-cd}
\usepackage{xypic}

\newtheorem{defn}{Definition}
\newtheorem{thm}{Theorem}

\newtheorem{lemma}{Lemma}

\newtheorem{prop}{Proposition}

\graphicspath{ {./images/} }

\journal{Discrete Mathematics}

\begin{document}

\begin{frontmatter}

\title{Loop Homology of Bi-secondary Structures}

%\title{Elsevier \LaTeX\ template\tnoteref{mytitlenote}}

\author[math2,bi]{Andrei C. Bura}
\ead{anbur12@vt.edu}
\author[bii]{Qijun He\corref{mycorrespondingauthor}}
\ead{qhe196@gmail.com}
\author[bii,math]{Christian M. Reidys}
\ead{duck@santafe.edu}


\cortext[mycorrespondingauthor]{Corresponding author}

\address[math2]{Department of Mathematics, Virginia Tech, 225 Stanger Street,
Blacksburg, VA 24061-1026}
\address[bi]{Biocomplexity Institute of Virginia Tech, 1015 Life Sciences Circle 
Blacksburg, VA 24061}
\address[bii]{Biocomplexity Institute and Initiative, University of Virginia, 995 Research Park Boulevard, Charlottesville, VA 22911}
\address[math]{Department of Mathematics, University of Virginia, 141 Cabell Dr, Charlottesville, VA 22903}

\begin{abstract}
  In this paper we compute the loop homology of bi-secondary structures.
  Bi-secondary structures were introduced by Haslinger and Stadler and are pairs of RNA secondary
  structures, i.e.~diagrams having non-crossing arcs in the upper half-plane.
  A bi-secondary structure is represented by drawing its respective secondary structures in the
  upper and lower half-plane.
  An RNA secondary structure has a loop decomposition, where a loop corresponds to a boundary
  component, regarding the secondary structure as an orientable fatgraph. The loop-decomposition
  of secondary structures facilitates the computation of its free energy and any two loops
  intersect either trivially or in exactly two vertices. In bi-secondary structures the intersection
  of loops is more complex and is of importance in current algorithmic work in bio-informatics and
  evolutionary optimization.
  We shall construct a simplicial complex capturing the intersections of loops and compute its
  homology. We prove that only the zeroth and second homology groups are nontrivial and furthermore
  show that the second homology group is free. Finally, we provide evidence that the generators of
  the second homology group have a bio-physical interpretation: they correspond to pairs of mutually
  exclusive substructures.
\end{abstract}


\end{frontmatter}

%\linenumbers

\section{Introduction}

RNA sequences are single stranded nucleic acids that, in difference to DNA, can form a plethora of structural
conformations. Over the last several decades, researchers have discovered an increasing number of important
roles for RNA \cite{darnell2011rna}. The folded structure of RNA is critically important to its function
\cite{holley1965structure} and has been extensively studied at the coarse grained level of base pairing
interactions. This leads to the notion of RNA secondary structures \cite{thirumalai2001early}, that
represent particular contact matrices and do not take into account the embedding in $3$-space
\cite{fresco1960some}.

The thermodynamic stability of a secondary structure is characterized by its free energy, and is computed
by summing the energy contribution of its loops \cite{gralla1973free,turner2009nndb}. Prediction of the
minimum free energy (i.e.~the most stable) secondary structure for a given sequence, is an important
problem at the most basic biological level \cite{pipas1975method}. 

The first mfe-folding algorithms for RNA secondary structures are due to
\cite{delisi1971prediction, fresco1960some, tinoco1971estimation}. Waterman studied the loop decomposition
and the recursive construction of secondary structures and derived the first dynamic programming (DP)
folding routines for secondary structures \cite{waterman1978secondary}. The DP routine facilitates
polynomial time folding algorithms \cite{nussinov1980fast,waterman1986rapid, zuker1984rna} and partition
function calculation \cite{mccaskill1990equilibrium}. In \cite{haslinger1999rna}, Haslinger and Stadler
extended the notion of secondary structures to bi-secondary structures in order to study pseudoknotted
structures, RNA structures exhibiting cross serial interactions \cite{taufer2008pseudobase++}.
Bi-secondary structures play furthermore a central role for studying sequences that can realize two,
oftentimes mutually exclusive, conformations, in the context of evolutionary transitions
\cite{reidys1997random} and in the study of RNA riboswitches, i.e.~sequences that exhibit two
stable configurations \cite{flamm2001design}. 

The partition function of structures w.r.t.~a fixed sequence has a dual: the partition
function of sequences compatible with a fixed structure \cite{busch2006info}. Partition function
and Boltzmann sampling have a variety of applications in sequence design
\cite{levin2012global, barrett2018efficient}, extracting
structural semantics \cite{barrett2017sequence} and to analyze mutational robustness \cite{he2018genetic}.  

RNA structures, viewed as abstract diagrams or trees, have been studied in enumerative combinatorics
\cite{waterman1978secondary,schmitt1994linear,hofacker1998combinatorics,haslinger1999rna}, algebraic
combinatorics \cite{jin2008combinatorics}, matrix-models \cite{orland2002rna,andersen2013topological}
and topology \cite{bon2008topological,andersen2013enumeration,huang2015shapes}.

In \cite{schmitt1994linear}, a bijection between linear trees and secondary structures was constructed.
This facilitated beautiful, explicit formulae for the number of secondary structures on $n$ vertices,
having exactly $k$ arcs. \\
Jin {\it et al.} \cite{jin2008combinatorics} enumerate $k$-non-crossing RNA structures, based on the
bijection given by Chen {\it et al.} \cite{chen2007crossings}, between $k$-non-crossing partial matchings
and walks in $\mathbb{Z}^{k-1}$ which remain in the interior of the Weyl-chamber $C_0$. The bijection
between oscillating tableaux and matchings originated from Stanley \cite{stanley1986enumerative} and was
generalized by Sundaram \cite{sundaram1990cauchy}.\\
Penner and Waterman connected RNA structures with topology by studying the space of RNA secondary
structures. They proved that the geometrical realizations of the associated complex of secondary structures
is a sphere \cite{penner1993spaces}. In \cite{bon2008topological}, Bon {\it et al.} presented a topological
classification of secondary structures, based on matrix models.\\
In the course of computing the Euler characteristics of the Moduli space of a curve, \cite{harer1986euler},
Harer and Zagier computed the generating function of the number of linear chord diagrams of genus $g$
with $n$ chords. Based on this line of work, Andersen {\it et al.} \cite{andersen2013topological},
enumerated the number of chord diagrams of fixed genus with specified numbers of backbones and chords.
Such an enumeration of chord diagrams provides the number of secondary structures of a given genus as
well as the number of cells in Riemann's moduli spaces for bordered surfaces. This was done by using
Hermitian matrix model techniques and topological recursions along the lines of \cite{ambjorn1993matrix}.  
Employing an augmented version of the topological recursion on unicellular maps of Chapuy
\cite{chapuy2011new}, Huang {\it et al} \cite{huang2015shapes} derived explicit expressions for the
coefficients of the generating polynomial of topological shapes of RNA structures and the generating
function of RNA structures of genus $g$. This lead to uniform sampling algorithms for structures of
fixed topological genus as well as a natural way to resolve crossings in pseudoknotted
structures \cite{huang2016topological}.

Bi-secondary structures emerge naturally in the context of evolutionary transitions, since they are closely
connected to sets of sequences, that are \emph{simultaneously} compatible with two structures
\cite{reidys1995neutral}. 
This paper is motivated by the dynamical programming (DP) routine of Huang \cite{Huang}, that is based on
sub-problems associated with sets of loops. The sub-problems were constructed incrementally by adding
one loop at a time, where subsequently added nucleotides affect the energy calculation if they appear
in multiple loops. This naturally leads one to consider intersections of loops and eventually to
introduce the \emph{nerve} of loops as a simplicial complex.

In this paper, we study the homology of bi-secondary structures \cite{haslinger1999rna}. We show that
for any bi-secondary structure $R$, we have only two nontrivial homology groups, $H_0(R)$ and $H_2(R)$.
The key to establish $H_1(R)=0$ is to establish in Lemma~\ref{L:delta} the existence of certain, spanning,
sub $1$-skeleta, whose existence follows from an inductive argument over the arcs of one of the secondary
structures. These skeleta give rise to specific trees, which in turn allow one to systematically process
elements of $\text{\rm Ker}(\partial_1)$.  We show that $H_0(R)\cong \mathbb{Z}$ and $H_2(R) \cong
\bigoplus_{k=1}^r \mathbb{Z}$, introducing the rank of $H_2(R)$ as a new invariant of the bi-secondary
structures. We show that $H_2(R)$ is free by showing that it is a subgroup of a free group, whose
freeness in turn is a consequence of Lemma~\ref{exposed} which guarantees the existence of exposed faces
of $3$-simplices. We then discuss the new invariant, observing that all RNA riboswitch sequences in
data-bases exhibit $\text{\rm rank}(H_2(R))=1$, seldomly assumed by random secondary structure pairs
and provide an outlook on future work.

%%%%
%%%%%%%%%%%%%%%%%%%%%%%%%%%%%%%%%%%%
%%%%
\section{Some basic facts}
%%%%
%%%%%%%%%%%%%%%%%%%%%%%%%%%%%%%%%%%%
%%%%

We shall begin by defining loops in an RNA secondary structure and then present results on its loop decomposition.

  An RNA diagram $S$ over $[n]$, is a vertex-labeled graph whose vertices are drawn on the horizontal axis
  and labeled by $[n]=\{1,\ldots,n\}$. An \emph{arc} $(i,j)$, is an ordered pair of vertices, which represents
  the base pairing between the $i$-th and $j$-th nucleotides in the RNA structure. Furthermore, each vertex can
  be paired with at most one other vertex, and the arc that connects them is drawn in the upper half-plane. We
  introduce two ``formal'' vertices associated with positions $0$ and $n + 1$, respectively, closing any diagram
  by the arc $(0, n + 1)$, called the rainbow. The set $[0,n+1]$ is called the diagram's \emph{backbone}, see
  Figure~\ref{fig1}.

  Let $S$ be an RNA diagram over $[n]$. Two arcs $(i, j)$ and $(p, q)$ are called \emph{crossing} if and only if
  $i < p < j < q$. $S$ is called a \emph{secondary structure} if it does not contain any crossing arcs. The arcs
  of $S$ can be endowed with a partial order as follows: $(k,l)\prec_S (i,j)\iff i < k < l< j$. We denote this by
  $(S,\prec_S)$ and call it the arc poset of $S$. Finally, an \emph{interval} $[i,j]$ on the backbone is the set
  of vertices $\{i,i+1,\ldots,j-1,j\}$.

  Let $S$ be a secondary structure over $[n]$. A \emph{loop} $s$ in $S$ is a subset of vertices, represented as
  a disjoint union of a sequence of contiguous blocks on the backbone of $S$, $s=\dot\bigcup_{i=1}^k [a_i,b_i]$,
  such that $(a_1,b_k)$ and $(b_i,a_{i+1})$, for $1\leq i\leq k-1$, are arcs and such that any other
  interval-vertices are unpaired. Let $\alpha_s$ denote the unique, maximal arc $(a_1,b_k)$ of
  the loop.

In this paper we shall identify a secondary structure with its set of loops. 
Let $S$ be a secondary structure over $[n]$ and $s=\dot\bigcup_{i=1}^k [a_i,b_i]$ a loop in $S$, then\\
  $(1)$ each unpaired vertex is contained in exactly one loop,\\
  $(2)$ $(a_1,b_k)$ is maximal w.r.t.~$\prec_S$ among all arcs contained in $s$, i.e.~there is a bijection between
        arcs and loops, mapping each loop to its maximal arc,\\
 $(3)$ the Hasse diagram of the $S$ arc-poset is a rooted tree $\text{Tr}(S)$, having the rainbow arc
        as root,\\
  $(4)$ each non-rainbow arc appears in exactly two loops. 

\begin{figure}[htbp]
    \centering
    \includegraphics[width=0.5\textwidth]{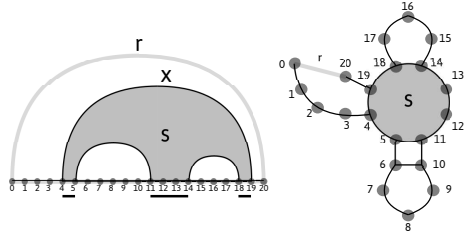}
    \caption{
      LHS: a secondary structure, $S$, and a distinguished loop $s=[4,5]\cup[11,14]\cup[18,19]$.
      $r$ is the rainbow arc and $\alpha_s=x$.
      RHS: $S$ represented as a planar $RNA$ molecule.}
    \label{fig1}
\end{figure}

\begin{prop}\label{3loops}
  Let $s$, $s'$ and $s''$ be three distinct loops in a secondary structure $S$. Then\\
  (1) $s\cap s' \cap s''=\varnothing$,
  (2) $s\cap s'\neq \varnothing$ implies $|s\cap s'|=2$.
\end{prop}

\begin{proof}
  Vertices of $S$ are either paired or unpaired. In the latter case, they are contained in
  exactly one loop. In the former, by construction, they are endpoints of arcs and contained in exactly two
  distinct loops. Hence, no vertex can be contained in three distinct loops. In case of $s\cap s'\neq \varnothing$
  the loops intersect in the endpoints of exactly one arc, which is maximal for exactly of of them, whence
  $|s\cap s'|=2$.
\end{proof}

%%%
%%%%%%%%%%%%%%%%%%%%%%%%%%%%%%%%%%%%%%%%%%%%%%%%%%%%%%%%%%%%%%%%%%%%%%%%%%%%%%%%%%%%%%
%%%
%\subsection{The nerve of a bi-secondary structure}
%%%
%%%%%%%%%%%%%%%%%%%%%%%%%%%%%%%%%%%%%%%%%%%%%%%%%%%%%%%%%%%%%%%%%%%%%%%%%%%%%%%%%%%%%%
%%%

In this section we introduce bi-secondary structures and their nerves. To this end we introduce the nerve
over a finite collection of sets:
%%%
%%%%%%%%%%%%%%%%%%%%%%%%%%%%%%%%%%%%%%%%%%%%%%%%%%%%%%%%%%%%%%%%%%%%%%%%%%%%%%%%%%%%%%
%%%

  Let $X=\{x_0,x_1,\ldots,x_m\}$ be a collection of finite sets. We call
  $Y=\{x_{i_0},\ldots,x_{i_d}\}\subseteq X$ a \emph{$d$-simplex} of $X$ iff
  $\bigcap_{k=0}^d x_{i_k}\neq \varnothing$. We set $\Omega(Y)=\bigcap_{k=0}^d x_{i_k}$ and
  refer to $\omega(Y)=|\Omega(Y)|\neq 0$ as the \emph{weight} of $Y$.
  Let $K_d(X)$ be the set of all $d$-simplices of $X$, then the \emph{nerve} of $X$ is
  $$
  K(X)=\dot\bigcup_{d=0}^{\infty} K_d(X)\subseteq 2^{X}.
  $$
  A $d'$-simplex $Y'\in K(X)$ is called a $d'$-face of $Y$ if $d'<d$ and $Y'\subseteq Y$. By construction,
  $K(X)$ is an abstract simplicial complex.
%%%
%%%%%%%%%%%%%%%%%%%%%%%%%%%%%%%%%%%%%%%%%%%%%%%%%%%%%%%%%%%%%%%%%%%%%%%%%%%%%%%%%%%%%%%%%
%%%
Let $S$ be a secondary structure over $[n]$. The geometric realization of $K(S)$, the nerve over the set
of loops of $S$, is a tree. By means of the correspondence between arcs and loops, this tree of loops is
isomorphic to $\text{Tr}(S)$. 
%%%
%%%%%%%%%%%%%%%%%%%%%%%%%%%%%%%%%%%%%%%%%%%%%%%%%%%%%%%%%%%%%%%%%%%%%%%%%%%%%%%%%%%%%%%%%
%%%
\begin{defn}
  Given two secondary structures $S$ and $T$ over $[n]$, we refer to the pair $R=(S,T)$ as a bi-secondary
  structure. Let $S\cup T$ be the loop set of $R$ and $K(R)=\dot\bigcup_{d=0}^{\infty} K_d(R)$ its nerve
  of loops.
\end{defn}
%%%
%%%%%%%%%%%%%%%%%%%%%%%%%%%%%%%%%%%%%%%%%%%%%%%%%%%%%%%%%%%%%%%%%%%%%%%%%%%%%%%%%%%%%%%%%
%%%
We represent the diagram of a bi-secondary structure $R=(S,T)$ with the arcs of $S$ in the upper half plane while the
arcs of $T$ reside in the lower half plane.
%%%
%%%%%%%%%%%%%%%%%%%%%%%%%%%%%%%%%%%%%%%%%%%%%%%%%%%%%%%%%%%%%%%%%%%%%%%%%%%%%%%%%%%%%%%%%
%%%
  Let $R=(S,T)$ be a bi-secondary structure with loop nerve $K(R)$. A $1$-simplex
  $Y=\{r_{i_0},r_{i_1}\}\in K_1(R)$ is called \emph{pure} if $r_{i_0}$ and $r_{i_1}$
  are loops in the same secondary structure and \emph{mixed}, otherwise.

%%%
%%%%%%%%%%%%%%%%%%%%%%%%%%%%%%%%%%%%%%%%%%%%%%%%%%%%%%%%%%%%%%%%%%%%%%%%%%%%%%%%%%%%%%%%%
%%%

Suppose $Y$ is a pure $1$-simplex in $K(R)$, then by Proposition~\ref{3loops} we have $\omega(Y)=2$, see
Figure~\ref{fig2}.

\begin{figure}[htbp]
    \centering
    \includegraphics[width=0.5\textwidth]{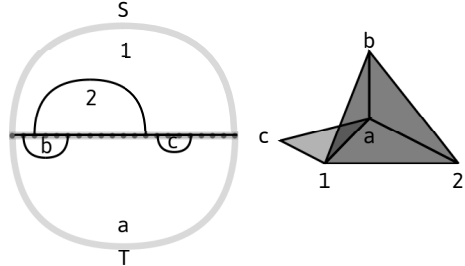}
    \caption{LHS: a bi-secondary structure $R=(S,T)$.
      RHS: the geometric realization of its loop nerve, $K(R)$. the $1$-simplices $\{c,1\}$ and
      $\{1,2\}$ are mixed and pure, respectively.}
    \label{fig2}
\end{figure}

%%%
%%%%%%%%%%%%%%%%%%%%%%%%%%%%%%%%%%%%%%%%%%%%%%%%%%%%%%%%%%%%%%%%%%%%%%%%%%%%%%%%%%%%%%%%%
%%%
\begin{lemma}\label{K2}
  Let $R=(S,T)$ be a bi-secondary structure with nerve $K(R)$. For any $Y\in K_2(R)$, exactly one of
  its three $1$-faces is pure, the other two being mixed.
  Furthermore, we have $\omega(Y)\le 2$.
\end{lemma}
%%%
%%%%%%%%%%%%%%%%%%%%%%%%%%%%%%%%%%%%%%%%%%%%%%%%%%%%%%%%%%%%%%%%%%%%%%%%%%%%%%%%%%%%%%%
%%%
\begin{proof}
  Let $Y=\{r'_0,r'_1,r'_2\}\in K_2(R)$ be a  $2$-simplex of $K(R)$. By Proposition~\ref{3loops},
  $\cap_{i=0,1,2} r'_i\neq \varnothing$ implies that not all three loops can be from the same structure.
  W.l.o.g.~suppose $r'_0,r'_1\in S$ and $r'_2\in T$. Certainly $Z=\{r'_0,r'_1\}$ is a pure $1$-face of $Y$
  and two other $1$-faces of $Y$ are by construction mixed since they contain $r'_3\in T$.
  For any $1$-face $Z'$ of $Y$, we have $\omega (Y)\leq \omega(Z')$ and Proposition~\ref{3loops}
  guarantees $\omega(Z)=2$, whence the lemma.
\end{proof}
%%%
%%%%%%%%%%%%%%%%%%%%%%%%%%%%%%%%%%%%%%%%%%%%%%%%%%%%%%%%%%%%%%%%%%%%%%%%%%%%%%%%%%%%%%%
%%%
\begin{lemma}\label{K3}
  Let $R=(S,T)$ be a bi-secondary structure with nerve $K(R)$ and let
  $Y=\{r_0,r_1,r_2,r_3\}\in K_3(R)$ be a $3$-simplex. Then we have\\
  (a) $Y=\{s_0,s_1,t_0,t_1\}$, where $s_0,s_1\in S$ and $t_0,t_1\in T$, \\
  (b) $Y$ has exactly two pure $1$-faces, $\{s_0,s_1\}$ and $\{t_0,t_1\}$, \\
  (c) $\omega(Y)\le 2$.
\end{lemma}
%%%
%%%%%%%%%%%%%%%%%%%%%%%%%%%%%%%%%%%%%%%%%%%%%%%%%%%%%%%%%%%%%%%%%%%%%%%%%%%%%%%%%%%%%%%
%%%
\begin{proof}
  Any $2$-simplex in the loop nerve is of the form $\{s,t,t'\}$ or $\{t,s,s'\}$ and $Y$ has
  the $2$-faces $\{r_0,r_1,r_2\}$, $\{r_1,r_2,r_3\}$, $\{r_0,r_2,r_3\}$ and $\{r_0,r_1,r_3\}$.
  In view of the $2$-simplex $\{r_0,r_1,r_2\}$, we can, w.l.o.g.~set $s_0=r_0$ $s_1=r_1$ and
  $t_0=r_2$. The $2$-simplex $\{r_0,r_1,r_3\}$ then implies that $r_3$ is a $T$-loop, whence
  we can set $t_1=r_3$ and $(a)$ follows. Assertion $(b)$ follows immediately from $(a)$.
  Finally, $(c)$, follows from $\cap_{i=0}^3 r_i\subset s_0\cap s_1$ and $\vert s_0\cap s_1\vert=2$.
\end{proof}

%%%
%%%%%%%%%%%%%%%%%%%%%%%%%%%%%%%%%%%%%%%%%%%%%%%%%%%%%%%%%%%%%%%%%%%%%%%%%%%%%%%%%%%%%%%
%%%
\begin{lemma}\label{2K3}
  Let $R=(S,T)$ be a bi-secondary structure and let $K(R)$ be its loop nerve. Then any pure
  $1$-simplex, $P$, appears as the $1$-face of at most two distinct $3$-simplices.
\end{lemma}
%%%
%%%%%%%%%%%%%%%%%%%%%%%%%%%%%%%%%%%%%%%%%%%%%%%%%%%%%%%%%%%%%%%%%%%%%%%%%%%%%%%%%%%%%%%
%%%
\begin{proof}
  W.l.o.g.~we may assume $P=\{s_0,s_1\}$, for some $s_0,s_1\in S$. If $P$ is a $1$-face of a $3$-simplex
  $Y$ then by Lemma~\ref{K3}, $Y=\{s_0,s_1,t_0,t_1\}$ for some $t_0,t_1\in T$. For any such $3$-simplex
  we have $\Omega(Y)\subset s_0 \cap s_1$, where $s_0\cap s_1=\{x,y\}$. Similarly $t_0\cap t_1=\{a,b\}$ and
  $\Omega(Y)\subset \{a,b\}$. In the case of $\{a,b\}=\{x,y\}$, $\{s_0,s_1\}$ is contained exclusively in the
  $3$-simplex $\{s_0,s_1,t_0,t_1\}$.
  Otherwise, we obtain two $3$-simplices $Y_x=\{s_0,s_1,t_0^x,t_1^x\}$ and $Y_y=\{s_0,s_1,t_0^y,t_1^y\}$ and in
  view of $\{x,y\}\cap t_0^x \cap t_1^x=\{x\}$ and $\{x,y\}\cap t_0^y \cap t_1^y=\{y\}$, both, $Y_x$ and $Y_y$
  contain $P$, see Figure~\ref{fig3}.
\end{proof}
\begin{figure}[htbp]
    \centering
    \includegraphics[width=0.5\textwidth]{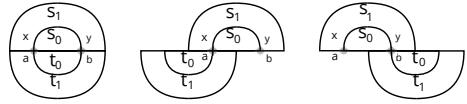}
    \caption{LHS: the case $\{a,b\}=\{x,y\}$. Center: the case of $Y_x=\{s_0,s_1,t_0^x,t_1^x\}$.
      RHS: the case of $Y_y=\{s_0,s_1,t_0^y,t_1^y\}$.}
    \label{fig3}
\end{figure}

%%%
%%%%%%%%%%%%%%%%%%%%%%%%%%%%%%%%%%%%%%%%%%%%%%%%%%%%%%%%%%%%%%%%%%%%%%%%%%%%%%%%%%%%%%%
%%%
\begin{defn}
  Let $K(X)=\dot\bigcup_{d=0}^\infty K_d(X)$ be an abstract simplicial complex and let $Y\in K_d(X)$ be a
  $d$-simplex. Let $Y'$ be a $(d-1)$-face of $Y$. We say $Y'$ is \emph{$Y$-exposed} if and only if no other
  $d$-simplices of $K$ contain $Y'$ as a $(d-1)$-face.
\end{defn}
%%%
%%%%%%%%%%%%%%%%%%%%%%%%%%%%%%%%%%%%%%%%%%%%%%%%%%%%%%%%%%%%%%%%%%%%%%%%%%%%%%%%%%%%%%%
%%%
\begin{lemma}\label{exposed}
  Let $R=(S,T)$ be a bi-secondary structure with loop-nerve $K(R)$. Then any $Y\in K_3(R)$ contains at least two
  $Y$-exposed $2$-faces.
\end{lemma}
%%%
%%%%%%%%%%%%%%%%%%%%%%%%%%%%%%%%%%%%%%%%%%%%%%%%%%%%%%%%%%%%%%%%%%%%%%%%%%%%%%%%%%%%%%%
%%%
\begin{proof}
  By Lemma~\ref{K3}, any $3$-simplex, $Y$, is of the form $Y=\{s_0,s_1,t_0,t_1\}$ and has exactly two pure
  $1$-faces, $P_1=\{s_0,s_1\}$ and $P_2=\{t_0,t_1\}$. We shall use $P_1$ to construct at least one specific,
  exposed $2$-face of $Y$.
  For $P_1$, $W_1=\{s_0,s_1,t_0\}$ and $W_2=\{s_0,s_1,t_1\}$ are the only two distinct $2$-faces, that contain $P_1$ as
  a pure $1$-face.
  In $K(R)$, $Y$ is the unique $3$-simplex that contains both $W_1$ and $W_2$ as $2$-faces.
  It thus remains to show that there cannot exist two distinct $3$-simplices $Y_1$ and $Y_2$ having $W_1$ and $W_2$ as
  a $2$-face, respectively. If this were the case, $Y,Y_1,Y_2$ were, by construction, three distinct $3$-simplices
  having $P_1$ as a pure $1$-face, which, in view of Lemma~\ref{2K3}, is impossible.
  Thus, either $W_1$ or $W_2$ is exposed in $Y$.
  We can argue analogously for $P_2$ and the lemma follows.
\end{proof}

%%%
%%%%%%%%%%%%%%%%%%%%%%%%%%%%%%%%%%%%%%%%%%%%%%%%%%%%%%%%%%%%%%%%%%%%%%%%%%%%%%%%%%%%%%%%%%%%%%%%
%%%
\section{Homology}
%%%
%%%%%%%%%%%%%%%%%%%%%%%%%%%%%%%%%%%%%%%%%%%%%%%%%%%%%%%%%%%%%%%%%%%%%%%%%%%%%%%%%%%%%%%%%%%%%%%%
%%%

In this section we consider the chain complex over the loop nerve $K(R)$ and compute its homology. We
will show that only the second homology group $H_2(R)$ is nontrivial and that $H_2(R)$ is free.
This produces a new invariant for bi-secondary structures, that provides insight into RNA riboswitch
sequences, i.e.~where a single sequence switches, depending on context, between mutually exclusive
structures.

Suppose we are given a bi-secondary structure $R=(S,T)$ and let $(T,\prec_T)$ and $(S,\prec_S)$ be
the posets of arcs on the secondary structures $T$ and $S$ respectively. $\prec_S$
and $\prec_T$ allow us to endow $R$ with the poset-structure:
$$
(R,\prec_R)=(T,\prec_T)\oplus (S,\prec_S),
$$
where $R=T\cup S$ and $\prec_R$ is given by
$$
r_1\prec_R r_2\Leftrightarrow
\begin{cases}
  r_1,r_2 \in T\; \text{\rm and } r_1\prec_T r_2 \\
  r_1,r_2 \in S\;  \text{\rm and } r_1\prec_S r_2 \\
  r_1\in S,\, r_2\in T
  \end{cases}
$$
Let us next choose a linear extension of $(R,\prec_R)$, $(R, \le)$, to which we refer to as the simplicial
order of the loop nerve. Any $d$-simplex, $Y\in K_d(R)$ becomes then the unique $d$-tuple
$Y=(r_0,r_1,\dots,r_d)$ where $r_0\le r_1\le \dots \le r_d$.

Let $R=(S,T)$ be a bi-secondary structure with loop nerve $K(R)$. Let $C_d(R)$ be the simplicial chain group
of dimension $d$ of $K(R)$. Let $Y =(r_0,r_1,\dots ,r_d)\in C_d(R)$ and $\partial_d:C_d(R)\longrightarrow
C_{d-1}(R)$ be the boundary map given by 
$$
  \partial _{d}(Y )=\sum _{i=0}^{d}(-1)^{i}\left(r_0,\dots ,r_{i-1},r_{i+1},\dots ,r_d\right).
$$
Let furthermore $H_d(R)=\text{\rm Ker}(\partial_{d})/ \text{\rm Im}(\partial_{d+1})$ be the $d$'th homology group of
the loop nerve of $R$. In the following we shall show

%%%
%%%%%%%%%%%%%%%%%%%%%%%%%%%%%%%%%%%%%%%%%%%%%%%%%%%%%%%%%%%%%%%%%%%%%%%%%%%%%%%%%%%%%%%%%%%
%%%
\begin{thm}\label{T:main}
The loop-nerve of a bi-secondary structure, $R$, has only the following nontrivial homology groups
\begin{eqnarray*}
  H_0(R) & = & \mathbb{Z} \\
  H_2(R) & = & \bigoplus_{k=1}^r\mathbb{Z}.
\end{eqnarray*}
\end{thm}
%%%
%%%%%%%%%%%%%%%%%%%%%%%%%%%%%%%%%%%%%%%%%%%%%%%%%%%%%%%%%%%%%%%%%%%%%%%%%%%%%%%%%%%%%%%%%%%
%%%

Let us begin proving Theorem~\ref{T:main} by first noting
%%%
%%%%%%%%%%%%%%%%%%%%%%%%%%%%%%%%%%%%%%%%%%%%%%%%%%%%%%%%%%%%%%%%%%%%%%%%%%%%%%%%%%%%%%%%%%%
%%%
\begin{lemma}\label{L:0}
$H_0(R)\cong \mathbb{Z}$.
\end{lemma}
%%%
%%%%%%%%%%%%%%%%%%%%%%%%%%%%%%%%%%%%%%%%%%%%%%%%%%%%%%%%%%%%%%%%%%%%%%%%%%%%%%%%%%%%%%%%%%%
%%%
\begin{proof}
By construction, the $1$-skeleton of $K(R)$ contains the two rooted trees associated to $S$ and $T$,
respectively. Their respective root-loops are connected by a $1$-simplex as both rainbows share the vertices $0$ and $n+1$. Thus any loop is path connected to a rainbow loop implying that any loop is, modulo boundaries,
equivalent to a rainbow loop. Hence the assertion follows.
\end{proof}

Let $t\in T$ be a loop, we set
$$
S(t)=\{s\in S\mid  \{s,t\}\in K_1(R)\}$$ $$T(t)=\{t'\in T\mid  t'\prec_T t,\nexists\:t''\in
T\text{ s.t. }\alpha_{t'}\prec_T \alpha_{t''}\prec_T \alpha_{t}, \{t,t'\}\in K_1(R)\},
$$ 
the sets of $S$ and $T$ neighbors of $t$, respectively. Let $R(t)=S(t)\cup T(t)$ and let
$\text{\rm Gr}(t)$ be %the graph induced on the $1$-skeleton, i.e.~let it be
the vertex induced
sub-graph of the $1$-skeleton in the geometric realization of $K(R)$, whose vertices are the loops
in $R(t)$. By construction, $\text{\rm Gr}(t)$, does not contain the loop $t$ as a vertex.

Let $R=(S,T)$ be a bi-secondary structure with loop nerve $K(R)$ and let $t\in T$ be a loop.
A connected, spanning sub-graph, $G(t)\le \text{\rm Gr}(t)$, in which each edge satisfies
$$
\{r_a,r_b\}\in G(t) \quad \Longrightarrow \quad \{r_a,r_b,t\}\in K_2(R),
$$
is called a $\Delta_t$-graph and we refer to its edges as $\Delta_t$-edges.

%%%
%%%%%%%%%%%%%%%%%%%%%%%%%%%%%%%%%%%%%%%%%%%%%%%%%%%%%%%%%%%%%%%%%%%%%%%%%%%%%%%%%%%%%%%%%%%
%%%
\begin{thm}\label{H1}
Let $R=(S,T)$ be a bi-secondary structure and $K(R)$ be its loop nerve, then $H_1(R)=0$.
\end{thm}
%%%
%%%%%%%%%%%%%%%%%%%%%%%%%%%%%%%%%%%%%%%%%%%%%%%%%%%%%%%%%%%%%%%%%%%%%%%%%%%%%%%%%%%%%%%%%%%
%%%
\begin{proof}
We shall inductively build $T$, arc by arc, from bottom to top and from left to right.\\
For the induction basis assume $T=\varnothing$, then, by construction, $K(R)=K(S)$ and the geometric
realization of its nerve is a tree, with edges between loops $p,q\in S$, whenever $p$ directly covers $q$
w.r.t.~$\prec_S$. Hence $H_1(R)=0$ and the induction basis is established.

For the induction step, the induction hypothesis stipulates $H_1(S,T)=0$. We shall show that
$H_1(R')=0$, where $R'=(S,T')$ and $T'$ is obtained from $T$ by adding the arc $\alpha_t$, the maximal
arc of the newly added loop $t$. We have the following scenario
\begin{equation}\label{E:embed}
\diagram
C_2(R')    \rto & C_1(R')    \rto & C_0(R')     \rto & 0 \\
C_2(R) \uto\rto & C_1(R)\uto \rto & C_0(R) \uto \rto & 0
\enddiagram
\end{equation}
where the vertical and horizontal maps are the natural embeddings and boundary homomorphisms,
respectively. \\
%%%
%%%%%%%%%%%%%%%%%%%%%%%%%%%%%%%%%%%%%%%%%%%%%%%%%%%%%%%%%%%%%%%%%%%%%%%%%%%%
%%%
{\it Claim $1$.}
$$
\text{\rm Ker}(\partial_1^{R'}) \subseteq \text{\rm Ker}(\partial_1^R) \oplus \text{\rm Im}(\partial_2^{R'}).
$$
%%%
%%%%%%%%%%%%%%%%%%%%%%%%%%%%%%%%%%%%%%%%%%%%%%%%%%%%%%%%%%%%%%%%%%%%%%%%%%%%
%%%
To prove the claim, we consider $\tau_0\in C_1(R')$: 
$$
\tau_0=\sum_{e_i\in K_1(R)} n_ie_i+ \sum_{e_j=\{r,t\}, r\in R(t)} n_j e_j,
$$
distinguishing any edges, that contain $t$, in the second term. The idea is to now process the edges
containing $t$ in a systematic way. To this end we first claim\\
%%%
%%%%%%%%%%%%%%%%%%%%%%%%%%%%%%%%%%%%%%%%%%%%%%%%%%%%%%%%%%%%%%%%%%%%%%%%%%%%
%%%
{\it Claim $2$.}
Let $R=(S,T)$ be a bi-secondary structure with nerve $K(R)$ and let $t$ be a $T$-loop, then, there
exists a $\Delta_t$-graph, $G(t)$.\\
%%%
%%%%%%%%%%%%%%%%%%%%%%%%%%%%%%%%%%%%%%%%%%%%%%%%%%%%%%%%%%%%%%%%%%%%%%%%%%%%
%%%
We shall give the proof of Claim $2$ by means of Lemma~\ref{L:delta}, below.\\

Given a $\Delta_t$-graph, any of its vertices can be employed as the root of a spanning $G(t)$-sub-tree
and we select the $\le$-maximum $G(t)$-vertex as root. Let $A(t)$ denote this rooted tree.
Any vertex, $r\in R(t)$, appearing in an edge $\{r,t\}$, occurs in $A(t)$ and any two $A(t)$-neighbors,
$\{r_1,r_2\}$ are in the boundary of the $2$-simplex $\{r_1,r_2,t\}$.

We examine now all $R(t)$-vertices in the following systematic way: starting with $A(t)$-leaves, pick $r_0$
and its unique, immediate, $A(t)$-ancestor, $r_1$.
We then have either\\
{\it Case $1$:} $r_0\le r_1$.\\
Then $(r_0,r_1)$ is a simplex and using that $\{r_0,r_1\}$ is a $\Delta_t$-edge, we are guaranteed that
$\{r_0,r_1,t\}$ is a $2$ simplex and
$$
\partial_2(r_0,r_1,t)=(r_1,t)-(r_0,t)+(r_0,r_1).
$$
We have a closer look at the sum of simplices $n_0(r_0,t)+n_1(r_1,t)$,
\begin{eqnarray*}
  n_0(r_0,t)+n_1(r_1,t) & = & n_0(r_0,t)+n_1(r_1,t)\pm n_0(r_0,r_1) \pm n_0(r_1,t) \\
                       & = & -n_0[(r_1,t)-(r_0,t)+(r_0,r_1)] +(n_0+n_1)(r_1,t)+n_0(r_0,r_1) \\ 
                       & = & -n_0\partial_2((r_0,r_1,t)) +(n_0+n_1)(r_1,t)+n_0(r_0,r_1).
\end{eqnarray*}
This produces on the RHS a boundary, a new term $(r_0,r_1)\in C_1(R)$, a modified coefficient
for the simplex $(r_1,t)$ and the term $(r_0,t)$ has become part of a boundary. \\
{\it Case $2$:} $r_1 \le r_0$.\\
Here $(r_1,r_0)$ is a simplex and
$$
\partial_2(r_1,r_0,t)=(r_0,t)-(r_1,t)+(r_1,r_0).
$$
Furthermore,
\begin{eqnarray*}
  n_0(r_0,t)+n_1(r_1,t) & = & n_0(r_0,t)+n_1(r_1,t)\pm n_0(r_0,r_1) \pm n_0(r_1,t) \\
                       & = & n_0[(r_0,t)-(r_1,t)+(r_0,r_1)] +(n_0+n_1)(r_1,t)-n_0(r_0,r_1) \\ 
                       & = & n_0\partial_2((r_1,r_0,t)) +(n_0+n_1)(r_1,t)-n_0(r_0,r_1). 
\end{eqnarray*}
On the RHS we, again, have a boundary, a new term $(r_0,r_1)\in C_1(R)$, a modified coefficient
for the simplex $(r_1,t)$ and the term $(r_0,t)$ has become part of a boundary. 

Iterating this procedure, we step by step transform simplices $\{r,t\}$ into boundaries, working
along the tree $A(t)$, from the leaves to the root. This finally produces the following expression
for $\tau_0$
$$
\tau_0=\epsilon_0+n_k(r_k,t)+\tau_k, 
$$
where $\epsilon_0\in \text{\rm Im}(\partial_2^{R'})$, i.e.~$\epsilon_0$ is a boundary, $r_k$ is the root of $A(t)$
and $\tau_k\in C_1(R)$. At this point we cannot proceed transforming $(r_k,t)$ into a boundary and shall
argue as follows: suppose $\tau_0\in \text{\rm Ker}(\partial_1^{R'})$. Then
$$
\partial_1^{R'}(\tau_0)=\partial_1^{R'}(\epsilon_0)+n_kt-n_k r_k+\partial_1^{R'}(\tau_k).
$$
Since $\epsilon_0\in \text{\rm Im}(\partial_2^{R'})$ we certainly have $\partial_1^{R'}(\epsilon_0)
=0$. By construction of the $\Delta_t$-graph $G(t)$, the $0$-simplex $\{t\}$ does not appear in
$\partial_1^{R'}(\tau_k)$, from which we conclude $n_k=0$. As a result we have $n_k r_k=0$ and
since $\partial_1^{R'}(\tau_k)=\partial_1^{R}(\tau_k)$, we have $0=\partial_1^{R'}(\tau_0)=\partial_1^{R}(\tau_k)$,
and as a result
$$
\tau_k\in \text{\rm Ker}(\partial_1^{R}).
$$
The induction hypothesis guarantees $H_1(R)=0$, i.e.~$\text{\rm Ker}(\partial_1^{R})= \text{\rm Im}(\partial_2^{R})$.
Hence $\tau_k\in \text{\rm Im}(\partial_2^{R})$, which in view of diagram~(\ref{E:embed}) implies $\tau_0\in
\text{\rm Im}(\partial_2^{R'})$ and we have proved
$\text{\rm Ker}(\partial_1^{R'})=\text{\rm Im}(\partial_2^{R'})$.
\end{proof}

It remains to show the proof of Claim $2$. To this end, let $r\in R$ be a loop with $\alpha_{r}=(a,b)$ and denote
$b(r)=b(\alpha_r)=a$ and $e(r)=e(\alpha_r)=b$.

%%%%
%%%%%%%%%%%%%%%%%%%%%%%%%%%%%%%%%%%%%%%%%%%%%%%%%%%%%%%%%%%%%%%%%%%%%%%%%%%%%%%%%%%%%%%%
%%%%
Let $s=\dot\bigcup_{i=1}^{k}[a_i,b_i]$ be a loop in a given secondary structure $S$.
We refer to the intervals $g_0(s)=[0,a_1]$, $g_{i}(s)=[b_i,a_{i+1}]$ for $1\le i\le k-1$
and $g_{k}(s)=[b_k,n+1]$, as the gaps of the loop $s$. We call $g_0(s)$ and $g_{k}(s)$ exterior
gaps and the rest interior gaps.  

%%%%
%%%%%%%%%%%%%%%%%%%%%%%%%%%%%%%%%%%%%%%%%%%%%%%%%%%%%%%%%%%%%%%%%%%%%%%%%%%%%%%%%%%%%%%%
%%%%

Claim $2$ now follows from
%%%%
%%%%%%%%%%%%%%%%%%%%%%%%%%%%%%%%%%%%%%%%%%%%%%%%%%%%%%%%%%%%%%%%%%%%%%%%%%%%%%%%%%%%%%%%
%%%%
\begin{lemma}\label{L:delta}
  Let $R=(S,T)$ be a bi-secondary structure with loop-nerve, $K(R)$, and let $t\in T$ be a loop, then,
  there exists a $\Delta_t$-graph, $G(t)$.
\end{lemma}
%%%%
%%%%%%%%%%%%%%%%%%%%%%%%%%%%%%%%%%%%%%%%%%%%%%%%%%%%%%%%%%%%%%%%%%%%%%%%%%%%%%%%%%%%%%%%
%%%%
\begin{proof}
Let $S(t)$ and $T(t)$ be the $S$ and $T$ neighbors of $t$ respectively.
We prove the lemma by induction on $N$, the number of non-rainbow arcs in $S$. To this end, let us
first consider the induction base case $N=0$.\\
As there are no arcs other than the rainbow, $\alpha_r$, we have $S(t)=\{r\}$. By construction,
$b(r)\in g_0(t)$ and $e(r)\in g_{k}(t)$, the exterior $t$-gaps. We make the Ansatz
$$
G(t)=\text{\rm Star}(r)=(R(t),\{\{r,t'\}|t'\in T(t)\}).
$$
By construction, $\text{\rm Star}(r)$ is a connected spanning sub-graph of $\text{\rm Gr}(t)$.
Furthermore, $\forall t'\in T(t)$ we have $b(r)< b(t)< b(t')< e(t')< e(t)< e(r)$. Hence
$$
r \cap t\cap t'=\{b(t'),e(t')\} \ne \varnothing
$$
and as a result $\{r,t,t'\}\in K_2(R(t))$. Thus, any edge $\{r,t'\}\in E(r)=\{\{r,t'\}|t'\in T(t)\}$
is a $\Delta_t$-edge and $\text{\rm Star}(r)$ is a $\Delta_t$-graph, establishing the induction basis,
see Figure~\ref{fig4}.

\begin{figure}[htbp]
    \centering
    \includegraphics[width=0.5\textwidth]{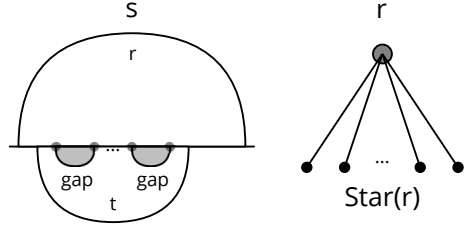}
    \caption{LHS: $S(t)=\{r\}$, RHS: $G(t)=Star(r)$.}
    \label{fig4}
\end{figure}

Let next $\overline{S}$ denote a secondary structure, having $N-1\ge 0$ arcs. By induction hypothesis,
for any such bi-secondary structure, $\overline{R}=(\overline{S},T)$ and $t\in T$, a $\Delta_t$-graph
exists. We will denote such a graph by $\overline{G}(t)$.

We shall prove the existence of a $\Delta_t$-graph as follows: first we identify and then remove a
distinguished non-rainbow arc $x\in S$. This gives us the bi-secondary structure
$\overline{R}=(\overline{S},T)$, for which the induction hypothesis applies, i.e.~a
$\Delta_t$-graph $\overline{G}(t)$ exists. We then reinsert the arc $x$ and inspect how to obtain
$G(t)$ from $\overline{G}(t)$.

Let $\text{\rm Exp}_t(S)$ be the set of non-rainbow $S$-arcs, $x$, having at least one $t$-exposed endpoint,
i.e.~either $b(x)$ or $e(x)$ are contained in $t$.

{\it Case $1$: $\text{\rm Exp}_t(S)\neq \varnothing$.}\\

\begin{figure}[htbp]
    \centering
    \includegraphics[width=0.75\textwidth]{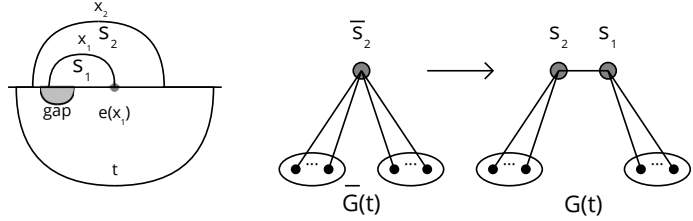}
    \caption{LHS: $x_1\in \text{\rm Exp}_t(S)$. RHS: the effect of reintroducing $x_1$, passing from
      $\overline{G}(t)$ to $G(t)$.}
    \label{fig5}
\end{figure}

Select $x_1\in \text{\rm Exp}_t(S)$. Let $s_1\in S(t)$
be the loop such that $x_1=\alpha_{s_1}$ and let $x_2$ be the arc, that directly covers $x_1$ w.r.t.~$
\prec_S$. Let $s_2\in S$ be the loop such that $\alpha_{s_2}=x_2$. W.l.o.g.~we may assume that $e(x_1)\in t.$ Clearly, $s_2\in S(t)$ since $e(s_1)
\in s_2\cap t$, see Figure~\ref{fig5}.

$x_1$-removal produces the secondary structure $\overline{S}$ and $\overline{R}=(\overline{S},T)$,
for which the induction hypothesis applies. Let $\overline{s_2}\in \overline{S}$ be such that
$\alpha_{\overline{s_2}}=x_2$. Then $\overline{s_2}\in \overline{S}(t)$ since, in absence of $x_1$,
$e(x_1)\in\overline{s_2}\cap t$. Hence $\overline{s_2}$ is a vertex in
$\overline{G}(t)=(\overline{R}(t),\overline{E})$.

Reinserting $x_1$ into $\overline{R}$ splits $\overline{s_2}$ into the two $S$-loops $s_1$ and $s_2$,
see Figure~\ref{fig5}.\\
We make the Ansatz
$$
G(t)=((\overline{R}(t)\setminus \{\overline{s_2}\})\cup \{s_1,s_2\},E),
$$
where
\begin{eqnarray*}
  E & = &
  (\overline{E}\setminus \{\{\overline{s_2},r'\}\mid r'\in \overline{R}(t)\}) \cup \{\{s_1,s_2\}\} \cup\\
& & \{\{s_1,r'\}\mid r'\in R(t)\setminus \{s_1,s_2\},\; s_1\cap r'\cap t\ne \varnothing\}\cup \\
& &  \{\{s_2,r'\}\mid r'\in R(t)\setminus \{s_1,s_2\},\; s_2\cap r'\cap t\ne \varnothing\}.
\end{eqnarray*}
Since $\overline{s_2}=s_1\cup s_2$ as sets, and $\overline{R}(t)\setminus \{\overline{s_2}\}=
R(t)\setminus \{s_1,s_2\}$, we have
$$
\{r'\in \overline{R}(t)\mid \{r',\overline{s_2}\} \in\overline{E} \} = \{r'\in \overline{R}(t)\setminus
\{\overline{s_2}\}
\mid \{r',s_1\}\in E \; \text{\rm or }\{r',s_2\}\in E\}.
$$
Accordingly, any $\overline{R}(t)$-vertex connected in $\overline{G}(t)$ to $\overline{s_2}$ is, when considered
in $R(t)$, connected to either $s_1$ or $s_2$. In view of $\{e(x_1)\}\subset (s_1\cap\ s_2\cap t)$, we can conclude that
$s_1$ and $s_2$ are connected by a $\Delta_t$-edge. This guarantees that $G(t)$ is a connected spanning sub-graph of
$\text{\rm Gr}(t)$.

{\it Case $2$: $\text{\rm Exp}_t(S)=\varnothing$.}\\
Having no arcs with exposed endpoints, for any loop $s\in S$, there exist $t$-gaps, containing $b(s)$ and
$e(s)$. Suppose first, there exists an arc $x$ having both endpoints in the same gap, see Figure~\ref{fig6}.
The associated loop, $s$, having $\alpha_s=x$, is not contained in $S(t)$. Upon inspection
$$
\overline{S}(t)=S(t) \; \text{\rm and } G(t)=\overline{G}(t),
$$
Hence the induction hypothesis directly implies the existence of $G(t)$.

\begin{figure}[htbp]
    \centering
    \includegraphics[width=0.25\textwidth]{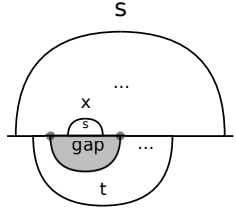}
    \caption{The case where $b(x)$ and $e(x)$ are contained in the same $t$-gap.}
    \label{fig6}
\end{figure}

It thus remains to discuss $S$-arcs, whose endpoints belong to distinct $t$-gaps, see Figure~\ref{fig7}.
We shall distinguish the following two scenarios:

{\it (a) $(S\setminus S(t))\setminus \{r\}\ne \varnothing$, where $\alpha_r$ is the rainbow arc.}\\
We shall show that the removal of an arc  $x=\alpha_s,s\in(S\setminus S(t))\setminus \{r\}$, will not affect $G(t)$,
aside from relabeling of a single vertex.
Since $s_1\notin S(t)$ we have $\overline{s_2}\notin \overline{S}(t)$ if and only if $s_2\notin S(t)$.
In this case we set $G(t)=\overline{G}(t)$ and the assertion is directly implied by the induction
hypothesis.
In case of $\overline{s_2}\in\overline{S}(t)$, $G(t)$ is obtained from $\overline{G}(t)$ by relabeling
$\overline{s_2}$ to $s_2$ exhibiting no other changes, see Figure~\ref{fig7}: 
$$
G(t)=((\overline{R}(t)\setminus \{\overline{s_2}\})\cup \{s_2\},E),
$$
where
$$
E=(\overline{E}\setminus \{\{\overline{s_2},r'\}\in \overline{E}\mid r'\in \overline{R}(t)\})
\cup\{\{s_2,r'\}\mid r'\in \overline{R}(t), \{\overline{s_2},r'\}\in \overline{E}\}
$$
and $G(t)$ is consequently a $\Delta_t$-graph.

\begin{figure}[htbp]
    \centering
    \includegraphics[width=0.75\textwidth]{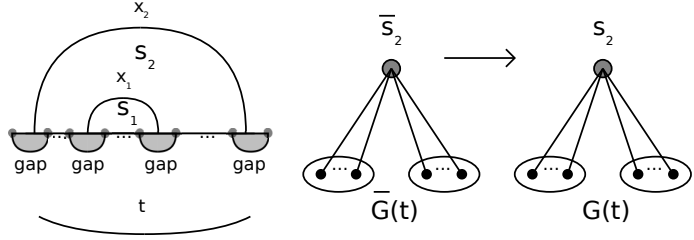}
    \caption{LHS: all $S$-arcs having their endpoints in distinct $t$-gaps. RHS: $(a)$, $\overline{s_2}\in \overline{S}(t)$,
    $G(t)$ is obtained by a relabelling of $\overline{G}(t)$.}
    \label{fig7}
\end{figure}

{\it (b) $(S\setminus S(t))\setminus \{r\}= \varnothing$, where $\alpha_r$ is the rainbow arc.}\\
We then have either $S\setminus S(t)=\varnothing$ or $S\setminus S(t)=\{r\}$.
In the latter case we select $x$ to be an arc, corresponding to a loop $s_1$, that is immediately
covered by $\alpha_r$. Let $s_2=r$. Since $r\notin S(t)$ we make the Ansatz
$$
G(t)=((\overline{R}(t)\setminus \{\overline{s_2}\})\cup \{s_1\},E),
$$
where
$$
E=(\overline{E}\setminus \{\{\overline{s_2},r'\}\in \overline{E}\mid r'\in \overline{R}(t)\})
           \cup\{\{s_1,r'\}\mid r'\in \overline{R}(t), \{\overline{s_2},r'\}\in \overline{E}\}.
$$
Accordingly, $G(t)$ is obtained from $\overline{G}(t)$ by relabeling $\overline{s_2}$ by $s_1$ and
$G(t)$ is a $\Delta_t$-graph, see Figure~\ref{fig8}.

It remains to analyze $S\setminus S(t)=\varnothing$, i.e.~all $S$-arcs are contained in $S(t)$,
where we recall we reduced the analysis to arcs whose endpoints belong to different $t$-gaps.

\begin{figure}[htbp]
    \centering
    \includegraphics[width=0.75\textwidth]{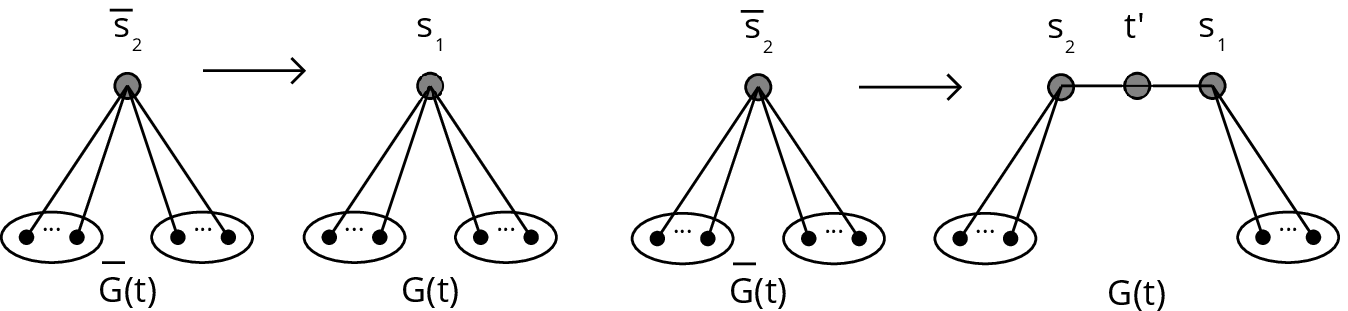}
    \caption{LHS: the case $S\setminus S(t)=\{r\}$. RHS: the case $S\setminus S(t)=\varnothing$.}
    \label{fig8}
\end{figure}

Suppose now all $S$-loops are contained in $S(t)$. Consider the set of all minimal arcs of $S$ w.r.t.~$\prec_S$. We claim there exists one such minimal arc, call it $\alpha_{s_1}$, such that its immediate cover w.r.t.~$\prec_S$, call it $\alpha_{s_2}$, is such that $s_2$ contains at least one of the endpoints of one of the $t$-gaps that contain one of the endpoints of $\alpha_{s_1}$. To show this we observe that if all $t$-gaps would have their endpoints inside loops corresponding to $\prec_S$-minimal arcs, then at least one arc that immediately covers such minimal arcs would not correspond to a loop in $S(t)$. Hence, there must be a loop $s_1$ with $\alpha_{s_1}$  minimal w.r.t.~$\prec_S$ and an arc $\alpha_{s_2}$ that immediately covers $\alpha_{s_1}$, such that $s_2$ contains one of the endpoints of a gap that contains $b(s_1)$ or $e(s_1)$.

Let us denote this gap by $h$. W.l.o.g.~we can assume $e(s_1)\in h$, see
Figure~\ref{fig9}. Then, the minimality of $s_1$ guarantees that $s_1$ contains the other endpoint of the gap
$h$. We shall now remove $x_1=\alpha_{s_1}$.

\begin{figure}[htbp]
    \centering
    \includegraphics[width=0.25\textwidth]{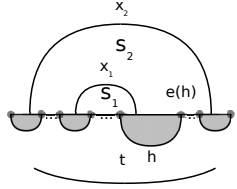}
    \caption{$x_1=\alpha_{s_1}$ is minimal w.r.t. $\prec_S$.}
    \label{fig9}
\end{figure}

We consider the loop $t'$ associated to $h$ and note that $s_1\cap t'\cap t\neq \varnothing$ as well as
$s_2\cap t'\cap t\neq \varnothing$. Accordingly, $t'$ connects $s_1,s_2$ in $R(t)$ by means of
$\Delta_t$-edges, see Figure~\ref{fig8} and we immediately obtain that
$$
G(t)=((\overline{R}(t)\setminus \{\overline{s_2}\})\cup \{s_1,s_2\},E),
$$
where 
\begin{eqnarray*}
  E & = & (\overline{E}\setminus \{\{\overline{s_2},r'\}\in \overline{E}\mid r'\in \overline{R}(t)\}) \cup\\
    &   & \{\{s_1,r'\}\mid r'\in R(t), s_1\cap r'\cap t\ne \varnothing\}\cup\\
    &   & \{\{s_2,r'\}\mid r'\in R(t), s_2\cap r'\cap t\ne \varnothing\},
\end{eqnarray*}
is a $\Delta_t$-graph for $R(t)$. This concludes the proof of the induction step and the lemma follows.
\end{proof}
%%%
%%%%%%%%%%%%%%%%%%%%%%%%%%%%%%%%%%%%%%%%%%%%%%%%%%%%%%%%%%%%%%%%%%%%%%%%%%%%%%%%%%%%%%%%%%%%%%%%
%%%

%%%
%%%%%%%%%%%%%%%%%%%%%%%%%%%%%%%%%%%%%%%%%%%%%%%%%%%%%%%%%%%%%%%%%%%%%%%%%%%%%%%%%%%%%%%%%%%%%%%%
%%%

Next we compute $H_2(R)$, 
%%%
%%%%%%%%%%%%%%%%%%%%%%%%%%%%%%%%%%%%%%%%%%%%%%%%%%%%%%%%%%%%%%%%%%%%%%%%%%%%%%%%%%%%%%%%%%%%%%%%
%%%
\begin{thm}\label{T:H2}
For any bi-secondary structure, $R=(S,T)$, with loop nerve $K(R)$, we have
$$
 H_2(R)\cong \bigoplus_{i=1}^k\mathbb{Z},
$$
i.e.~$H_2(R)$ is free of finite rank.
\end{thm}
%%%
%%%%%%%%%%%%%%%%%%%%%%%%%%%%%%%%%%%%%%%%%%%%%%%%%%%%%%%%%%%%%%%%%%%%%%%%%%%%%%%%%%%%%%%%%%%%%%%%
%%%
\begin{proof}
{\it Claim $1$. $\text{\rm Im}(\partial_3) \cong C_3(R)$, i.e.~$\text{\rm Im}(\partial_3)$ is a free
Abelian group and freely generated by $P=\{\partial_3(Y_i)\mid Y_i \; \text{\rm is a $3$-simplex} \}$.}

Claim $1$ is a consequence of two facts: (a) $C_4(R)=0$ and (b) $H_3(R)=0$, both of which we prove below.
It is obtained as follows: $C_4(R)=0$ guarantees $\text{\rm Im}(\partial_4)=0$, which in view of
$0=H_3(R)=\text{\rm Ker}(\partial_3)/ \text{\rm Im}(\partial_4)$ implies $\text{\rm Ker}(\partial_3)=0$. This in turn implies
that $\partial_3$ is an embedding, i.e.~$\text{\rm Im}(\partial_3)\cong C_3(R)$, whence $\text{\rm Im}(
\partial_3)$ is a free Abelian group. $P$ certainly generates $\text{\rm Im}(\partial_3)$ and a
$\mathbb{Z}$-linear combination
$$
\sum_{j}\lambda_j \partial_3(Y_j)=\partial_3(\sum_{j}\lambda_j Y_j)=0
$$
means that $\sum_{j}\lambda_j Y_j\in \text{\rm Ker}(\partial_3)$. Since the latter is trivial we arrive at
$$
\sum_{j}\lambda_j Y_j=0
$$
which implies $\lambda_j=0$, for any $j$ appearing in this sum. This shows that the $P$-elements are
$\mathbb{Z}$-linear independent.

Let $Y_i\in K_3(R)$, $0\le i\le k$ denote the generators of $C_3(R)$. Lemma~\ref{exposed} guarantees
that each $3$-simplex $Y$ has at least two $Y$-exposed $2$-faces. Hence, to each generator $Y_i\in
K_3(R)$ there correspond at least two generators of the free group $C_2(R)$ that appear as terms only
in the image $\partial_3(Y_i)$.
Let us write
$$
\partial_3(Y_i)=\sum_{u(i)} U_{u(i)} + \sum_{c(i)}C_{c(i)},
$$
distinguishing exposed, signed and covered, signed $2$-faces of $Y_i$, i.e.~we consider the
sign, induced by the boundary map, to be part of $U_{u(i)}$ and $C_{c(i)}$, respectively.
In particular, for any $u(i)$, there exists an unique $r$, such that we have either $U_{u(i)}=
+ Z_r$ or $U_{u(i)}= - Z_r$ where $Z_r$ is a generator of $C_2(R)$.

{\it Claim $2$. $C_2(R)/\text{\rm Im}(\partial_3)$ is free.}

We consider $\overline{C_2(R)}=C_2(R)/\text{\rm Im}(\partial_3)$ as a $\mathbb{Z}$-module and suppose $X$
is a torsion element of order $n$ in $\overline{C_2(R)}$. Then we can represent $X$ as
$$
X = \sum_r \lambda_r Z_r + \text{\rm Im}(\partial_3),
$$
where, w.l.o.g.~we assume that all $\lambda_r\neq 0$.
Since $X$ is a torsion element, we have $nX=0$ in $\overline{C_2(R)}$, i.e.~ 
\begin{eqnarray*}
  n (\sum_r \lambda_r Z_r ) & = &  \sum_{i=1}^k \alpha_i \partial_3(Y_i) \\
  & = & \sum_{i=1}\alpha_i (\sum_{u(i)} U_{u(i)}) + \sum_{i=1}\alpha_i (\sum_{c(i)} C_{c(i)}) 
\end{eqnarray*}
where $\lambda_r,\alpha_i\in \mathbb{Z}$ are unique nonzero integer coefficients. Clearly, each unique
signed $2$-face, $U_{u(i)}$ of the RHS corresponds to a unique generator $Z_{r(U_{u(i)})}$ and hence,
irrespective of the particular choice of the $u(i)$ and the sign of $U_{u(i)}$, we obtain
for any $i$ of the sum on the RHS
$$
n \lambda_{r(U_{u(i)})} = \alpha_i,
$$
only depending on the index $i$. This is an equation in $\mathbb{Z}$ and hence implies
$\alpha_i\equiv 0\mod n$. Accordingly, we derive
$$
 (\sum_r \lambda_r Z_r )  =   \sum_{i=1}^k \frac{\alpha_i}{n} \partial_3(Y_i),
$$
which means $X\in \text{\rm Im}(\partial_3)$, since $\frac{\alpha_i}{n}\in\mathbb{Z}$, i.e.~$X=0$.
By transposition we have proved that $X\neq 0$ in $\overline{C_2(R)}$ implies for any
$n\in \mathbb{N}$, $nX\neq 0$ in $\overline{C_2(R)}$, whence $\overline{C_2(R)}$ is free
and Claim $2$ is proved.

As a result, $\text{\rm Ker}(\partial_2)/\text{\rm Im}(\partial_3)$ is, as a subgroup of the
free group $\overline{C_2(R)}$, itself free and the theorem is proved. 
\end{proof}

It remains to show $C_d(R)=0$ for $d\ge 4$ and $H_3(R)=0$.

%%%
%%%%%%%%%%%%%%%%%%%%%%%%%%%%%%%%%%%%%%%%%%%%%%%%%%%%%%%%%%%%%%%%%%%%%%%%%%%%%%%%%%%%%%%%%%%%%
%%%
\begin{lemma}\label{d>3}
Let $R=(S,T)$ be a bi-secondary structure with loop-nerve $K(R)$ and let $d\ge 4$,
then $K_d(R)=\varnothing$, $C_d(R)=0$ and $H_d(R)=0$. 
\end{lemma}
%%%
%%%%%%%%%%%%%%%%%%%%%%%%%%%%%%%%%%%%%%%%%%%%%%%%%%%%%%%%%%%%%%%%%%%%%%%%%%%%%%%%%%%%%%%%%%%%%
%%%
\begin{proof}
For any $Y=\{r_{i_0},\ldots,r_{i_d}\}\in K_d(R)$ for some $d\geq 4$ we have
$\Omega(Y)=\bigcap_{k=0}^{d} r_{i_k}\neq \varnothing$.
Since $d\geq 4$, $|Y|\ge 5$, whence at least three loops $r'_0,r'_1,r'_2\in Y$ are contained
in the same secondary structure, which is a contradiction to Proposition~\ref{3loops}, which stipulates
that three loops of one secondary structure intersect only trivially. 
\end{proof}

Next we show that $H_3(R)=0$.
%%%
%%%%%%%%%%%%%%%%%%%%%%%%%%%%%%%%%%%%%%%%%%%%%%%%%%%%%%%%%%%%%%%%%%%%%%%%%%%%%%%%%%%%%%%%%%%%%%%
%%%
\begin{thm}\label{H3}
Let $R=(S,T)$ be a bi-secondary structure with loop nerve $K(R)$, then $H_3(R)=0$.
\end{thm}
%%%
%%%%%%%%%%%%%%%%%%%%%%%%%%%%%%%%%%%%%%%%%%%%%%%%%%%%%%%%%%%%%%%%%%%%%%%%%%%%%%%%%%%%%%%%%%%%%%%
%%%
\begin{proof}
Consider
$$
X=\sum_{Y_i\in K_3(R)} n_{i} Y_i \in C_3(R)\in \text{\rm Ker}(\partial_3),
$$
then
\begin{eqnarray*}
  \partial_3(X) & = & \sum_{Y_i\in K_3(R)} n_{i} \partial_3(Y_i)\\
                & = & \sum_i n_i (\sum_{u(i)} U_{u(i)}) + \sum_i n_i (\sum_{c(i)} C_{c(i)}),
\end{eqnarray*}
where the $U_{u(i)}$ and $C_{c(i)}$ are the signed exposed and covered $2$-faces of $Y_i$, respectively.
Since we have at least two unique exposed $2$-faces and, by assumption, $\partial_3(X)=0$, we conclude
that for any $i$ of $X=\sum_{Y_i\in K_3(R)} n_{i} Y_i$ we have $n_i=0$. Therefore $X=0$ and
$\text{\rm Ker}(\partial_3)$ contains no nontrivial elements, whence $H_3(R)=0$.
\end{proof}

%%%
%%%%%%%%%%%%%%%%%%%%%%%%%%%%%%%%%%%%%%%%%%%%%%%%%%%%%%%%%%%%%%%%%%%%%%%%%%%%%%%%%%%%%%%%%%%%%%%
%%%
\section{Discussion} 

In the previous section, we showed that $H_2(R)$ is non-trivial and free, leading to a novel observable
for the pair of secondary structures $(S,T)$, namely the rank of $H_2(R=(S,T))$, $r(H_2(R))$.
We shall see that the generators of $r(H_2(R))$ represent key information about the ``switching sequence''
\cite{breaker2012riboswitches}, a segment of the sequence, that engages w.r.t.~each respective structure
in a distinct, mutually exclusive fashion.

It is well known from experimental work that native riboswitch pairs, ncRNAs, exhibit two distinct,
mutually exclusive, stable secondary structures \cite{serganov2013decade}. We analyzed all nine riboswitch
sequences contained in the Swispot database \cite{barsacchi2016swispot} and observed that $r(H_2(R))=1$.
In Figure ~\ref{conflicting} we illustrate the connection between a $H_2(R)$-generator and pairs of mutually
exclusive substructures.
\begin{figure}[htbp]
    \centering
    \includegraphics[width=0.5\textwidth]{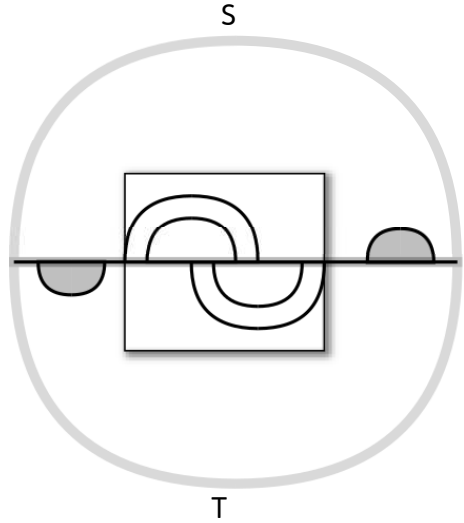}
    \caption{$H_2(R)$-generators and mutually exclusive structure pairs:
      the two helices (boxed) are mutually exclusive, while the two substructures (shaded) are not.
      The former two, together with the two rainbows correspond to a generator of $H_2(R)$.}
    \label{conflicting}
\end{figure}
The ranks, $r(H_2(R))$, for uniformly sampled structures pairs, are displayed in Figure ~\ref{random},
showing that $6.7\%$ of the uniform random pairs exhibit $r(H_2(R))=1$.

\begin{figure}[htbp]
    \centering
    \includegraphics[width=0.75\textwidth]{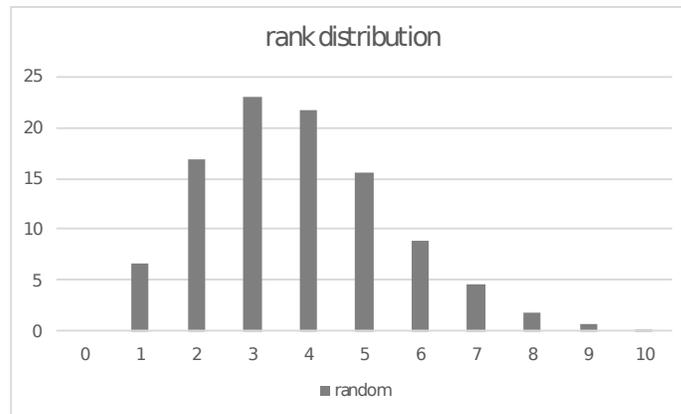}
    \caption{$r(H_2(R))$ for uniformly random structure pairs: $r(H_2(R))$ ($x$-axis) and
      the relative frequencies ($y$-axis).}
    \label{random}
\end{figure}

As for future work, the complexity analysis and optimal scheduling problems arising from the work
of Huang \cite{Huang} suggest to consider a graded version of the homologies, developed here.
Let $t\ge 1$ be an integer and $R=(S,T)$ be a bi-secondary structure.
We set $K_d^t(R)=\{Y\in K_d(R)| \omega(Y)\ge t\}$, the set of $d$-simplices of weight
at least $t$. We define $C_d^t(R)$ to be the free abelian group generated by $K_d^t(R)$, i.e.~the
chain group of rank $d$ and weight at least $t$. It is easy to see then that $C_d(R)=C_d^1(R)$ for
all $d\ge 0$ and that $C_d^{t+1}(R)\leq C_d^t(R)$ for all $t\ge 1$ and all $d\ge 0$.  We can naturally
define boundary operators for these groups in terms of restrictions of our original boundary maps as
$\partial_d^t:C_d^t(R)\xrightarrow[]{}C_{d-1}^t(R)$ with $\partial_d^t=\partial_d|_{C_d^t(R)}$.  As such,
we obtain a $t$-parametric sequence of nerves $\{K^t(R)\}_{t\ge 1}$ each of which gives rise to its
$t$-labelled homology sequence. Tracking the persistence of homology group generators across the
newly obtained homological $t$-spectrum gives rise to a more granular analysis of the structure of
the complete nerve~\cite{zomorodian2005computing}. This analysis represents a version of persistent
homology, pioneered by Edelsbrunner and by Gunnar Carlson \cite{edelsbrunner2002persistence,gunnar2009topology,gunnarzomo2005barcodes} and is of central importance
for designing an optimal loop-removal schedule in \cite{Huang}.

We extend the homology analysis to planar interaction structures \cite{andersen2012topology}.
Due to the fact that the physical $5'-3'$ distance for $RNA$ strands is in general very
small~\cite{yoffe2010ends}, the formation of an interaction structure is connected to the alignment of
two discs, each representing the respective, circular backbone. That is, interpreting the two circles
corresponding to two interacting secondary structures $S_1,S_2$ to be $\partial(D(0,1))$, the boundaries
of unit disks in $\mathbb{C}$. These boundaries contain distinguished points that correspond to the
paired vertices in the secondary structures. The connection between interaction structure and
disc-alignment leads one to consider one disc being acted upon by M{\"o}bius transforms. This action
is well defined since the M{\"o}bius maps of the disc map the boundary to itself and, being holomorphic,
cannot introduce crossings. Different alignments are then captured by these automorphisms and give
rise to a spectrum of homologies, as introduced in this paper.

\section{Declarations of interest}
None.

\section{Acknowledgments}
We gratefully acknowledge the comments from Fenix Huang. Many thanks to Thomas Li, Ricky Chen and Reza Rezazadegan for discussions.

\section*{References}

%\bibliography{homology_DM.bib}

\end{document}